\documentclass[12pt]{amsart}

\usepackage{amssymb}

\usepackage{enumerate}
\usepackage{comment}

\usepackage{graphicx}

\makeatletter
\@namedef{subjclassname@2010}{%
  \textup{2010} Mathematics Subject Classification}
\makeatother

\newtheorem{theorem}{Theorem}[section]
\newtheorem{corollary}[theorem]{Corollary}
\newtheorem{lemma}[theorem]{Lemma}
\newtheorem{proposition}[theorem]{Proposition}

\theoremstyle{definition}
\newtheorem{definition}[theorem]{Definition}

\newtheorem{example}[theorem]{Example}

\numberwithin{equation}{section}

\begin{document}


\baselineskip=17pt



\title[On directional derivatives for cone-convex functions]{On directional derivatives for cone-convex functions}

\author[K. Le\'sniewski]{Krzysztof Le\'sniewski}
\address{Faculty of Mathematics and Information Science\\ Warsaw University of Technology\\
00-662 Warszawa, Poland, ul. Koszykowa 75}
\email{k.lesniewski@mini.pw.edu.pl}
\date{}

\begin{abstract}
We investigate the relationship between the existence of directional derivatives for cone-convex functions with values in a Banach space $Y$ and isomorphisms between $Y$ and $c_0.$

\end{abstract}

\subjclass[2010]{Primary 	46B10; Secondary 	46B20}

\keywords{directional derivative, cone isomorphism, convex mappings, cone convex mappings}

\maketitle

\section{Introduction}
Cone convexity of vector-valued functions plays a similar role as convexity for real-valued functions (e.g. local Pareto minima are global etc.).  Directional derivatives of cone-convex functions are used in constructing descent methods for vector-valued mappings (see, e.g. \cite{drummond}). However, in contrast to the real-valued case the existence of directional derivative for cone-convex vector-valued functions is not automatically guaranteed.

If we assume that cone $K\subset Y$ is normal and $Y$ is weakly sequentially complete Banach space, then directional derivatives for $K$-convex functions $F:X \rightarrow Y$ always exists (see \cite{valadier}). The question is what can we say about converse implication. 

In Theorem 5.2 of \cite{lesniewski} it was proved that the existence of directional derivatives of a cone-convex mapping $F:X\rightarrow Y$ defined on linear space $X$  with values in a real Banach space $Y$  implies the  weak sequential completness of the  space $Y$. 

In the proof of Theorem 5.2 of \cite{lesniewski} we construct a cone $K$ and a  $K$-convex function such that the directional derivative does not exist at $0$ for some $h\in X\setminus \{0\}$. However in \cite{lesniewski}, cone $K$ is not normal.

In the present paper we relate the existence of  directional derivatives for $K$-convex functions $F:X\rightarrow Y$, where cone $K$ is normal to the existence of isomorphisms between image space $Y$ and  $c_0$. 

Isomorphisms of a Banach space $Y$ and the space $c_0$ have beed investigated e.g. in \cite{bessaga, laz, rosc0}. A sequence $\{b_j\}$ in a Banach space $Y$ is strongly summing (s.s.) if $\{b_j\}$ is a weak Cauchy basic sequence and for any scalars $\{c_j\}$ satisfying  $\sup_n\|\sum\limits_{j=1}^nc_jb_j\|<+\infty$, the series  $\sum c_j$ converges. In \cite{rosc0} Rosenthal proved the following theorem.

\begin{theorem}(Theorem 1.1. of \cite{rosc0}) A  Banach space $Y$ contains no subspace isomorphic to $c_0$ if and only if every non-trivial weak-Cauchy sequence in $Y$ has a (s.s.)-subsequence. 
\end{theorem}

In \cite{bessaga} Bessaga and Pe\l czy\'nski proved \emph{Bessaga-Pe\l czy\'nski $c_0$-theorem}.
\begin{theorem}[Theorem 5 of \cite{bessaga}]
A Banach space $Y$ does not contain a subspace isomorphic to $c_0$ if and only if every series $\sum\limits_{k=1}^\infty x_k $ such that $\sum\limits_{k=1}^\infty |\langle x_k, x^* \rangle|<\infty$ $\forall x^* \in Y^*$ is unconditionally convergent.
\end{theorem} 
In \cite{laz, mey} we can find interesting result for Banach lattices.
\begin{theorem}
\label{dlakrat}
A Banach lattice does not contain a subspace isomorphic to $c_0$ if and only if it is a weakly sequentially complete.
\end{theorem}

The organization of the paper is as follows. In Section 2 we present basic notions and facts about cone-convex functions. Section 3 is devoted to normal cones in Banach spaces. In Section 4 we present basic  construction of cone-convex functions (cf. \cite{lesniewski}) which is used in our main result.  Section 5 contains the main result.

\section{Notations and preliminaries}
Let $X$ be a linear space over reals. Let $Y$ be a normed space  over reals and let $Y^*$ be the norm dual of $Y$. 
\begin{definition}
	Let $A\subset X$. The function $F: A \rightarrow Y$ is \emph{directionally differentiable} at
	$x_0 \in A $ in the direction $h\neq 0$ such that $x_0+th \in A$ for all  $t$ sufficiently small
	 if  the limit
	
	\begin{equation}
	F'(x_0;h): = \lim_{t\downarrow 0 } \frac{F(x_0+th)- F(x_0)}{t}
	\end{equation}
	exists. The element $F'(x_0;h)$ is called the \emph{directional derivative} of $F$ at $x_{0}$ in the direction $h.$
\end{definition}

A nonempty subset $K$ of $Y$ is called a \emph{cone}  if $\lambda K \subset K$ for every $\lambda \ge 0 $ and $K+K\subset K.$
The relation $x\le_K y$ $(y\ge_K x)$ is defined as follows
$$
x\le_K y\  (y\ge_K x)\ \ \Leftrightarrow y-x \in K.
$$

The \emph{dual cone} \cite{jahn} of a cone $K$ is defined as  
\begin{equation}
\label{dualny1}
K^*=\{ y^* \in Y^*: y^*(y) \ge 0 \ \ \mbox{ for all } y \in K\}.
\end{equation}
In the space $c_0:=\{ x=(x_1,x_2,\dots), x_i \in \mathbb{R}, \lim\limits_{i\rightarrow \infty} x_i = 0\},$ the cone $c_0^+,$ 
 $$
c_0^+:=\{ x=(x_1,x_2,\dots) \in c_0 : \ x_i \ge 0,\ i=1,2,\dots\},
$$
is a closed convex pointed (i.e. $K\cap (-K)=\{0\}$). 
Cone $c_0^+$ is generating in $c_0$ i.e. $c_0=c_0^+ - c_0^+.$
\begin{example}
Let us show that 
$(c_0^+)^*=l_1^+:=\{ g=(g_1,g_2,\dots) \in l_1: g_i \ge 0, \ i=1,2\dots \}.$
Since $c_0^*=l_1$ for any $y^*=(g_1,g_2,\dots )\in (c_0^+)^*, $ we have $\sum\limits_{i=1}^\infty |g_i| < \infty.$
Since every $e_i=(0,0,\dots,\underbrace{1}_i,0,\dots),$ $i=1,2,\dots $ is an element of $c_0^+$  we have
$$y^*(e_i)=g_i, \ i=1,2,\dots\  .$$  By the definition of dual cone \eqref{dualny1}, we get
 $
g_i\ge 0 \mbox{ for all } i=1,2,\dots $ .
On the other hand, for any $x\in (c_0^+)^*$  and any $y^*=(g_1,g_2,\dots )\in l_1^+$ we have 
$y^*(x)\ge 0.$ 
\end{example}
\begin{definition}
\label{Kwypuklosc}
Let $K\subset Y$ be a cone. Let $A\subset X$ be a convex set.
We say that a function $F:X\rightarrow Y$ is \emph{$K$-convex} on $A$ if $\forall \,x, y \in A $ and $\forall\, \lambda \in [0,1]$ 
$$
\lambda F(x)+ (1-\lambda) F(y) - F(\lambda x + (1-\lambda ) y ) \in K. 
$$
\end{definition}
Some properties of $K$-convex functions can be found in e.g. \cite{lesniewski, drummond, pennanen}. 
The following characterization is given in \cite{pennanen} for finite dimensional case. 
\begin{lemma}[Lemma 3.3 of \cite{lesniewski}]
	\label{lemma-convex}
	 Let $A\subset X$ be a convex subset of $X$. Let $K\subset Y$ be a  closed convex  cone  and let  $F:X\rightarrow Y$ be a function. The following conditions are equivalent.
\begin{enumerate}
\item
$
\mbox{The function }F \mbox{ is $K$-convex on }A.$
\item $\mbox{For any } u^{*}\in K^{*} \mbox{, the composite function }  u^{*}(F):A\rightarrow\textbf{R}
\mbox{ is convex}$.
\end{enumerate}
\end{lemma}
\section{Normal Cones}
In a  normed space $Y$ a cone $K$ is normal (see \cite{polyrakiss}) if there is a number $C>0$ such that 
$$
0\le_K x\le_K y \Rightarrow \|x\|\le C \|y\|.
$$
Some useful characterizations of normal cones are given  in the following lemmas.
\begin{lemma}{\cite{peressini}}
Let $Y$ be an ordered topological vector space with positive cone $K.$ The following assertions are equivalent. 
\begin{itemize}
\item $K\subset Y$ is normal.
\item For any two nets $\{ x_\beta \ : \ \beta \in I \} $ and $\{ y_\beta :\ \beta \in I  \}$, if $0\le_K x_\beta \le_K y_\beta $ for all $\beta \in I$ and $\{ y_\beta \}$ converges to $0,$ then $\{ x_\beta \} $ converges to $0$.
\end{itemize} 
\end{lemma}


For lattice cones in Riesz spaces defined as in \cite{aliprantis} we get the following Lemma.
\begin{lemma}[Lemma 2.39 of \cite{aliprantis}]
Every lattice cone in Riesz space is normal closed and generating.
\end{lemma}

In some infinite dimensional spaces there are pointed generating cones which are not normal.

\begin{example}[Example 2.41 of \cite{aliprantis}]
Let $Y=C^1[0,1]$ be the real vector space of all continuously differentiable functions on $[0,1]$ and let cone $K$ be defined as
$$
K:= \{ x\in C^1[0,1]: x(t) \ge 0 \mbox{ for all } t\in [0,1]\}.
$$
Let us consider the norm 
$$
\| x\|= \|x\|_\infty + \|x'\|_\infty,
$$
where $x'$ denotes the derivative of $x\in Y.$
Cone $K$ is closed and generating but it is not normal. Let $x_n:= t^n$ and $y_n:=1$, we have $0\le_K x_n \le_K 1$. There is no constant $c>0$ such that the inequality 
$\|x_n\|=\|t^n\|_\infty + \|nt^{n-1}\|_\infty= n+1\le c=c\|y\|$ holds for all $n$.
\end{example}
Some interesting results (see e.g. \cite{mcarthur,singer,polyrakis}) for closed convex cones are using the concept  of a basis of a space.
\begin{definition}[Definition 1.1.1 of \cite{albiac}]
\label{basic}
A sequence $\{x_n\}\subset Y$ in an infinite-dimensional Banach space $Y$ is said to be a  \emph{basis} of $Y$ if for each $x\in Y$ there is a unique sequence of scalars $\{a_n\}_{n\in \mathbb{N}}$ such that 
$$
x=\sum\limits_{n=1}^{\infty} a_n x_n.
$$

\end{definition}
For basis $\{x_n\}$ we can define the cone associated to the basis $\{x_n\}.$

\begin{definition}[Definition 10.2 of \cite{singer}]
Let $\{x_n\}$ be a basis of a Banach space $Y$. The set 
$$
K_{\{x_n\}}:= \{ y\in Y : y=\sum\limits_{i=1}^{\infty}  \alpha_i x_i \in Y \ : \alpha_i \ge 0 , i=1,2,\dots  \}
$$
is called the cone associated to the basis $\{x_n\}.$
\end{definition}
Cone $K_{\{x_n\}}$ is a closed and convex and coincides with the cone generated by $\{x_n\}$ i.e. it is the smallest cone containing $\{x_n\}$.

For a basis $\{x_n\}$ of a Banach space $Y$ functionals $\{x_n^*\}$ are called \emph{biorthogonal functionals} if $x_k^*(x_j)=1$ if $k=j$, and $x^*_k(x_j)=0$ otherwise, for any $k,j\in \mathbb{N}$ and $x=\sum\limits_{i=1}^\infty x_i^*(x)x_i$ for each $x\in X.$ The sequence $\{x_n, x_n^*\}$ is called the \emph{ biorthogonal system}.

It is easy to see that $\{e_i\},$ where $e_i=\underbrace{(0,\dots,1,0,\dots)}_i$, $i=1,2,\dots$ is a basis for $c_0$ and  $\{e_i^*\}$, $e_i^*:=e_i$  are biorthogonal functionals.

We also have 
$c_0^+ = K_{\{e_i\}}$ and $(c_0^+)^*=l_1^+= K_{\{e_i^*\}}.$
\begin{definition}[Definition 3.1.1 of \cite{albiac}]
A basis $\{x_n\}$ of a Banach space $Y$ is called \emph{unconditional} if for each $x\in Y$ the series $\sum\limits_{n=1}^\infty x^*_n(x)x_n$ converges unconditionally.
\end{definition}
A basis $\{x_n\}$ is \emph{conditional} if it is not unconditional.
A sequence $\{x_n\}\subset X$ is \emph{complete} (see \cite{terenzi}) if $\overline{span\{x_n\}}=X.$
\begin{theorem}[Proposition 3.1.3 of \cite{albiac}]
\label{al}
Let $\{x_n\}$ be a complete sequence in a Banach space $Y$ such that $x_n \neq 0$ for every $n$. Then the following statements are equivalent. 

\begin{enumerate}
\item
$\{x_n\}$ is an unconditional basis for $Y$.
\item
$ \exists \ C_1 \ge 1$\ $\forall N\ge 1$ \   $\forall c_1,\dots,c_N$\ \ $\forall \varepsilon_1,\dots,\varepsilon_N=\pm 1,$
\begin{equation}
\label{nier}
\|\sum\limits_{n=1}^N \varepsilon_n c_n x_n\| \le C_1 \|\sum\limits_{n=1}^N  c_n x_n\|.
\end{equation}

\end{enumerate}
\end{theorem}
First example of conditional basis for $c_0$ was given by Gelbaum \cite{gelbaum}.
\begin{example}
Basis $\{x_n\}$ defined as $x_n:=(1,1,\dots,1,0,\dots)=\sum\limits_{i=1}^n e_i$, ($\{e_i\}$ is the canonical basis for $c_0$), $n=1,2,\dots$ is a conditional basis for $c_0.$ 
All calculations can be found in Example 14.1 p.424 \cite{singer}.
\end{example}
\begin{example}
\label{exex}
Let us show that $\{b_i\}\subset c_0$ defined as
\begin{equation}
\label{baz}
b_i=\frac 1 i e_i, \ i=1,2,\dots,
\end{equation}
($\{e_i\}$ is the canonical basis for $c_0$)
is an unconditional basis.
\end{example}
Since $\|x\|=\sup\limits_i |x^i|$ for $x=(x^1,x^2,\dots)\in c_0$, inequality \eqref{nier} is satisfied with $C_1=1$ 
\begin{align*}
& \|\varepsilon_1 c_1 (1,0,\dots)+\dots+\varepsilon_N c_N (0,0 ,\dots, \frac{1}{N}, 0, \dots)\| = \\
& \|(\varepsilon_1 c_1,\dots, \varepsilon_N c_N \frac{1}{N}, 0, \dots)\|=\|(c_1,\dots, c_N \frac{1}{N}, 0, \dots)\|.
\end{align*}

In \cite{mcarthur, singer} we can find a characterization of normal cones in terms of unconditional bases. 
\begin{theorem}[Theorem 16.3 of \cite{singer}]
\label{mc}
Let $\{x_n, f_n\}$ be a  complete biorthogonal system in a Banach space $Y.$ The following are equivalent.
\begin{enumerate}
\item
$\{x_n\}$ is an unconditional basis for $Y$.
\item
$K_{\{x_n\}}$ is normal and generating.
\end{enumerate}
\end{theorem}

\begin{example}
\label{exex}
Cone $K_{\{x_n\}}\subset c_0$, where 
$$
x_n=(0,0\dots,0,\underbrace{-1}_n, 1,0,0\dots), n=1,2\dots
$$
is generating pointed and not normal in $c_0$.
\end{example}

Now let us present some facts about cone isomorphisms.
\begin{definition}[\cite{casini, polyrakis}]
Let $X$ and $Y$ be normed spaces ordered by cones $P\subset X$ and $K\subset Y,$ respectively.
We say that $P$ is \emph{conically isomorphic} to  $K$ if there exists an additive, positively homogeneous, one-to-one map $i$ of $P$ onto $K$ such that $i$ and $i^{-1}$ are continuous in the induced topologies. Then we also say that $i$ is a conical isomorphism of $P$ onto $K$.

\end{definition}

\begin{proposition}
\label{roz}
Let $X$ be a linear space and let $Y, Z$ be Banach spaces. Let $P$ and $K$ be convex cones in $Z$ and $Y$, respectively.
Let function $F:X \rightarrow P$ be a $P$-convex. If there exists a conical isomorphism $i: P \rightarrow K$,  where cone $P$ is generating  in $Z$, then the function $\bar{F}: X \rightarrow K$, where 
$$
\bar{F}:=i \circ F
$$
is a $K$-convex function.
\end{proposition}
\begin{proof}
By Theorem 4.4 of \cite{casini}, in view of the fact that $P$ is a generating cone in $Z$, the conical isomorphism $i: P\rightarrow K$ can be extended to the function  $G : P \rightarrow K-K$ defined as 
$$G(x)=i(x^1)-i(x^2), \mbox{ where } x=x^1-x^2, x^1, x^2\in P.$$  Function $G: Z \rightarrow K-K$ is linear. Indeed, let us take $x,y\in Z$, since $P$ is generating $x=x^1-x^2$, $y=y^1-y^2$, where $x^1,x^2,y^1,y^2 \in P.$
$$
\begin{array}{cc}
G(x+y)&=G(x^1-x^2+y^1-y^2)=i(x^1+y^1)-i(x^2+y^2)=\\
&=i(x^1)-i(x^2)+i(y^1)-i(y^2)=G(x)+G(y).
\end{array}
$$
Let us take $\lambda<0.$ We have
$$
G(\lambda x)= G(\lambda x^1- \lambda x^2)=i(-\lambda x^2)-(-\lambda x^1)= -\lambda i (x^2) + \lambda i(x^1) = \lambda G(x).
$$
For $\lambda \ge 0$ the calculations are analogous.

Now let us take $x_1, x_2 \in X$ and $\lambda \in [0,1].$
By Definition \ref{Kwypuklosc} and the linearity of $G,$
$$
\begin{array}{cl}
F(\lambda x_1 + (1-\lambda)x_2 ) \le_P \lambda F(x_1) + (1-\lambda) F(x_2)\ &  i.e. \\
\lambda F(x_1) + (1-\lambda) F(x_2) - F(\lambda x_1 + (1-\lambda)x_2 ) & \in P.\\
\end{array}
$$
By the definition of $G,$
$$
\begin{array}{ll}
G(\lambda F(x_1) + (1-\lambda) F(x_2) - F(\lambda x_1 + (1-\lambda)x_2 ) )=\\
i(\lambda F(x_1) + (1-\lambda) F(x_2) - F(\lambda x_1 + (1-\lambda)x_2 ) )\in K.\\
\end{array}
$$
Furthermore, 
$$
\begin{array}{ll}
G(\lambda F(x_1) + (1-\lambda) F(x_2) - F(\lambda x_1 + (1-\lambda)x_2 ) )=\\
\lambda G( F(x_1)) +  (1-\lambda)G( F(x_2)) -G( F(\lambda x_1 + (1-\lambda)x_2 ) )=\\
\lambda i( F(x_1)) +  (1-\lambda)i( F(x_2)) - i( F(\lambda x_1 + (1-\lambda)x_2 ) ) \ge_{K} 0.
\end{array}
$$
The latter inequality is equivalent to
$$
\bar{F}(\lambda x_1 + (1-\lambda) x_2) \le_K \lambda \bar{F}(x_1) + (1-\lambda) \bar{F}(x_2)
$$
which completes the proof.
\end{proof}
\begin{proposition}
\label{izomnorm}
Let  $K\subset Y$ be closed convex and normal  cone in a Banach space $Y$. If $i: Y \rightarrow Z$ is an isomorphism, then $i(K)$ is a closed convex and normal cone in $Z$.
\end{proposition}
\begin{proof}
Let us take 
\begin{equation}
\label{eq1704}
0\le_{i(K)} z^1_n\le_{i(K)} z^2_n \ \ n=1,2,\dots
\end{equation}
for some sequences $z_n^1,z_n^2 \in Z$ and let $\lim\limits_{n\rightarrow \infty}z_n^2\rightarrow 0.$ We want to show that $\lim\limits_{n\rightarrow \infty}z^1_n\rightarrow 0.$ By \eqref{eq1704}, we have $z^2_n-z^1_n \in i(K).$ By assumption $i^{-1}: Z \rightarrow Y$ is continuous linear and onto i.e. we get,
$$
K \ni i^{-1} (z^2_n-z_n^1) = i^{-1} (z^2_n)- i^{-1}(z^1_n).
$$
It means that $i^{-1}(z^2_n)\ge_{K} i^{-1}(z^1_n).$
Since $z_2\rightarrow 0$ we get  $i^{-1}(z^2_n)\rightarrow 0.$ From the fact that $K$ is normal $\lim\limits_{n\rightarrow 0}i^{-1}(z^1_n)= 0\equiv \lim\limits_{n\rightarrow 0}z^1_n= 0.$
By Banach Open Mapping Theorem, cone $i(K)$ is closed.
\end{proof}
From Proposition \ref{izomnorm} and Example \ref{exex} we get the following corollary.
\begin{corollary}
Every Banach space isomorphic to $c_0$ contains a cone which is closed convex and generating but not normal.
\end{corollary}

Let $J$ be a James space i.e. $J:=\{ x=(x_n)_{n\in \mathbb{N}} : \ \lim\limits_{n\rightarrow 0}x_n =0\}$ with the norm $\|x\|= \sup \{  (\sum\limits_{i=1}^n (x_{m_{2i-1}}- x_{m_{2i}})^2)^{\frac 1 2}:   0=m_0<m_1<\dots<m_{n+1}\}<\infty.$
An interesting result is the fact that James space does not contain an isomorphic copy of $c_0$ or $l_1$. 

In \cite{bessaga} Pe\l czy\'nski and Bessaga  proved the following theorem.
\begin{theorem}[Theorem 6.4 of \cite{mcarthur, bessaga}]
A separable Banach space having the space $J$ of James as a subspace (e.g. $C[0,1]$) does not have an unconditional basis.
\end{theorem}
It is easy to find not normal cones in infinite dimensional Banach spaces (see e.g. \cite{singer}).

\section{Useful Constructions of convex functions}
In this section we recall a constructions of some convex functions introduced in \cite{lesniewski}.
Let us start with the following lemma.
\begin{lemma}[Lemma 4.2 of \cite{lesniewski}]
\label{lemat1}
Let  $\{a_{m}\},\{t_m\}\subset\textbf{R}$ be sequences with $\{t_m\}$ decreasing. 
Function $g:\textbf{R}\rightarrow \textbf{R}$ defined as
$$
g(r):=\sup_m f_m(r), 
$$
where 
$$f_m(x):=
a_{m}+\frac{x-t_m}{t_{m+1}-t_{m}}(a_{m+1}-a_{m}) \ \ \mbox{ for } x\in \textbf{R}.
$$
is convex and 
$$
 g(t_{m})=a_{m}.
 $$
 if and only if 
\begin{equation}
\label{nierwonoscwaznaa}
\frac{a_{m+1}-a_{m}}{t_{m+1}-t_{m}} \ge \frac{a_{m+2}-a_{m+1}}{t_{m+2}-t_{m+1}} \ \ \ \mbox{ for } m\in\textbf{N}.
\end{equation}
\end{lemma}
\begin{proof}
$\Leftarrow$ this proof has beed presented in \cite{lesniewski}.

$\Rightarrow$ Let us assume that 
$$
\frac{a_{m+1}-a_{m}}{t_{m+1}-t_{m}}< \frac{a_{m+2}-a_{m+1}}{t_{m+2}-t_{m+1}} \ \ \ \mbox{ for some } m\in\textbf{N}.
$$
We get
$$
(a_{m+1}-a_{m})\left(1- \frac{x-t_m}{t_{m+1}-t_{m}}\right)+ \frac{x-t_{m+1}}{t_{m+2}-t_{m+1}}(a_{m+2}-a_{m+1})>0
$$
and
$f_{m+1}(x)>f_m(x)$ for $x\ge t_{m+1}.$ Let us take $x=t_m\ge t_{m+1}$ we get
$$
f_{m+1}(t_m) > f_m(t_m)=a_m, 
$$
which is contrary to $g(t_m)=a_m.$
\end{proof}

Let $Y$ be a Banach space and $\{y_i\}$ be an arbitrary sequence of elements of $Y$. Let $\{t_i\}$ be a sequence of positive reals tending to zero. 

Let $\bar{F}: X_h:= \{x\in X : x=\beta h, \beta \ge 0\} \rightarrow Y $ be a function defined as in  in \cite{lesniewski}, i.e. for $r>0$ 
\begin{equation}
\label{et1}
\bar{F}(rh) := \sum\limits_{i=1}^\infty \bar{F}_i(rh),
\end{equation}
where 
$\bar{F}_i: \{x\in X : x=\beta h, \beta \ge 0\} \rightarrow Y$ is defined as
$$\bar{F}_1(rh):=\begin{cases}
y_{1}t_{1}+\frac{r-t_{1}}{t_{2}-t_{1}}(y_{2}t_{2}-y_{1}t_{1})&  \ t_{2}< r\\
0 & \  r <t_{2}\end{cases}$$
and for $i\ge 2$
$$\bar{F}_i(rh):=\begin{cases}
y_{i}t_{i}+\frac{r-t_{i}}{t_{i+1}-t_{i}}(y_{i+1}t_{i+1}-y_{i}t_{i})&  \ t_{i+1}< r\le t_{i}\\
0 & \  r \notin (t_{i+1}, t_{i}].\end{cases}.$$
Observe that for $r=t_k$ we have $\bar{F}(t_kh)=\bar{F}_k(t_kh)=y_kt_k.$

The following proposition  is a simple consequence of Lemma \ref{lemat1}.
\begin{proposition}
\label{prop1}
Let $K\subset Y$ be a closed convex cone with the dual $K^*\subset Y^*.$ Let $\{y_i\}\subset Y$ be a sequence in $Y$. The function $\bar{F}$ defined by \eqref{et1}  with $t_k = \frac 1 k$, $k=1,2,\dots,$ is $K$-convex on $X_h$ if and only if 
\begin{equation}
\label{glen}
y^*(2y_{k+1}- y_k-y_{k+2}) \le 0 \ \ \mbox{ for all } y^*\in K^*,\  k=1,2,\dots\ .
\end{equation}
\end{proposition}

\begin{proof}
By Lemma \ref{lemat1},
the function $\bar{F}$ is $K$-convex on $X_h$ if and only if inequality \eqref{nierwonoscwaznaa} holds, i.e.
 $$
 \frac{a_{k+1}-a_{k}}{t_{k+1}-t_{k}} \ge \frac{a_{k+2}-a_{k+1}}{t_{k+2}-t_{k+1}},
 $$

where  $ t_k=\frac 1 k, a_k:= y^*(\bar{F}(t_kh))$  and $ y^*\in K^*$.  We get 
  \begin{align*}
 \frac{\frac{1}{k+2}y^*(y_{k+2})-\frac{1}{k+1} y^*(y_{k+1}) }{\frac{1}{k+2}-\frac{1}{k+1}} &\le \frac{
 \frac{1}{k+1}y^*(y_{k+1})-\frac 1 k  y^*(y_{k}) }{\frac{1}{k+1}-\frac{1}{k}} & \equiv\\
 k y^*(y_{k+1})-(k+1)  y^*(y_{k}) &\le  (k+1)  y^*(y_{k+2}) -(k+2) y^*(y_{k+1}) & \equiv \\
 (k+1) y^*(2y_{k+1}-y_k-y_{k+2}) & \le 0,\  k=1,2,\dots  & \equiv\\
   y^*(2y_{k+1}-y_k-y_{k+2}) & \le 0,\   k=1,2,\dots \ .
 \end{align*}
 \end{proof}
  \begin{proposition}
  \label{bies}
  Let $\{y_i\}\subset Y$ be a sequence in a Banach space $Y$.
  If $\{b_k\}\subset Y$ defined as
  $$
2y_{k+1}- y_k-y_{k+2} =:b_k,\  k=1,2,\dots\ ,
$$
 forms an unconditional basis in $Y$, the function $\bar{F}: X\rightarrow Y$ defined by \eqref{et1}  with $t_k=\frac 1 k, k=1,2,\dots ,$ is $(-K_{\{b_k\}})$-convex, where
$$
K_{\{b_k\}}:= \{y\in Y : y=\sum\limits_{i=1}^\infty a_i b_i, a_i \ge 0,\  i=1,2,\dots\}
$$

is a closed generating and normal cone in $Y$.

  \end{proposition}
\begin{proof}
In view of Proposition \ref{al}, the cone
$$
K_{\{b_k\}}=\{ y\in Y: y=\sum\limits_{i=1}^{\infty}  \alpha_i b_i \in Y \ : \alpha_i \ge 0 , i=1,2,\dots  \}
$$ 
is normal and generating in $Y.$
Let us observe that dual cone is defined as
$$
K^*:=-K_{\{b_i^*\}},
$$
where $\{b_i,b_i^*\}$ is the biorthogonal system.
Inequality \eqref{glen} is satisfied because 
$$
y^*(b_k)= -\alpha_k \le 0
$$
for some $\alpha_k \in \mathbb{R}.$
By Proposition \ref{prop1}, the function $\bar{F}$ is $(-K_{\{b_k\}})$-convex.
\end{proof}

\section{Main result}
Now we are ready to formulate our main result.
\begin{theorem}
\label{twierdzenie}
Let $X$ be a linear space and $Y$ be a Banach space. If for every  closed convex and normal cone $K\subset Y$ and 
 for every $K$-convex function $F:X \rightarrow Y$ there exist directional derivatives for every $h\in X$ at $x_0=0,$ then there is no subspace in Y isomorphic to $c_0.$
\end{theorem}
\begin{proof}
We proceed by contradiction. We assume that there exists a subspace $Z\subset Y$ isomorphic to $c_0$, i.e. there is a continuous linear mapping $i: c_0 \rightarrow Z.$
Our aim is to construct a cone-convex function $F: X \rightarrow Y$ which does not possess the directional derivative at 0 for some $h\in X, h\neq 0.$

Let us take $h\in X$, $h\neq 0 $ and $t_k:=\frac{1}{k}, $ $k=1,2,\dots\ .$
Let $\{y_k\}\subset c_0$ be a sequence in $c_0.$
 
The proof will be in two steps.

Step 1. By using \eqref{et1}, let us construct a cone-convex function $\bar{F}:X \rightarrow c_0$ such that 
$$
\bar{F}(\frac 1 k h)=y_k \frac 1 k, \ k=1,2,\dots\ .
$$

Let  $\bar{F}: \{x\in X : x=\beta h, \beta \ge 0\} \rightarrow c_0$ be defined as follows. For $r>0$ 
$$
\bar{F}(rh) := \sum\limits_{i=1}^\infty \bar{F}_i(x),$$
where 
$\bar{F}_i: \{x\in X : x=\beta h, \beta \ge 0\} \rightarrow c_0$ is defined as
$$\bar{F}_1(x):=\begin{cases}
y_{1}t_{1}+\frac{r-t_{1}}{t_{2}-t_{1}}(y_{2}t_{2}-y_{1}t_{1})&  \ t_{2}< r\\
0 & \  r <t_{2}\end{cases}$$
and for $i\ge 2$
$$\bar{F}_i(x):=\begin{cases}
y_{i}t_{i}+\frac{r-t_{i}}{t_{i+1}-t_{i}}(y_{i+1}t_{i+1}-y_{i}t_{i})&  \ t_{i+1}< r\le t_{i}\\
0 & \  r \notin (t_{i+1}, t_{i}]
\end{cases}.$$

In view of Proposition \ref{bies} we need to find a sequence $\{y_k\}$ such that

1. the sequence $\{b_k\}$ defined as 
\begin{equation}
\label{rownanie}
2y_{k+1}-y_k-y_{k+2}=b_k,\  k=1,2,\dots
\end{equation}
is an unconditional basis of $c_0,$ 

2. $\{y_k\}$ is not weakly convergent.

Let us prove by induction that 
\begin{equation}
\label{ind}
y_k=(k-1)y_2-(k-2)y_1 - \sum\limits_{i=1}^{k-2} (k-i-1)b_i.
\end{equation}

Let us assume that equality \eqref{ind} holds for all $n\le k$.
By \eqref{rownanie} we have

\begin{align*}
& y_{k+1}=2y_k-y_{k-1}-e_{k-1}&=\\
&2[(k-1)y_2-(k-2)y_1-\sum\limits_{i=1}^{k-2}(k-i-1)b_i]-y_{k-1}-b_{k-1} &=\\
& 
2[(k-1)y_2-(k-2)y_1-\sum\limits_{i=1}^{k-2}(k-i-1)b_i]-\\
&\underbrace{[(k-2)y_2-(k-3)y_1-\sum\limits_{i=1}^{k-3}(k-i-2)b_i]}_{y_{k-1}}-b_{k-1}&=\\
& ky_2-(k-1)y_1-\sum\limits_{i=1}^{k-1}(k-i)b_i
\end{align*}
which proves \eqref{ind}.

Let $\{b_i\}\subset c_0$ be the unconditional basis defined  in Example \ref{exex}, i.e. $b_i=(0,0,\dots,\frac{1}{i},0,\dots), i=1,2,\dots\ .$

Let $K:=-K_{\{b_k\}}.$
By Proposition \ref{bies}, cone $K$ is normal and generating in $c_0.$

Let us take $y_1=(0,0,\dots)$ and $y_2=(1,\frac 1 2 , \frac 1 3 ,\dots).$

By \eqref{ind}, we get
\begin{align*}
y_3=&(1,1,\frac 2 3, \frac 2 4, \frac 2 5, \frac 2 6,\dots),\\
\dots\\
y_k=&(1,1,\dots,\underbrace{1}_{k-1}, \frac{k-1}{k}, \frac{k-1}{k+1},\dots).
\end{align*}

Moreover, observe that $\{y_k\}$ is weak Cauchy.
Let $y^*=(f_1,f_2,\dots)\in l_1.$ We have
$$
y^*(y_{k+1}-y_k)=f_{k+1} \frac{2}{k+1} + f_{k+2}\frac{1}{k+2} + f_{k+3}\frac{1}{k+3}+\dots \xrightarrow{k\rightarrow \infty} 0.
$$
Since $y_k \rightarrow (1,1,\dots)\notin c_0$ sequence $\{y_k\}$ is not weakly convergent.

Step 2. By assumption, there exists an isomorphism $i$ between $c_0$ and the subspace $Z\subset Y.$ 

Let us define $F:X\rightarrow Y$ by the formula $F:=i\circ \bar{F}.$
From Proposition \ref{roz}, the function $F$ is $i(K)$-convex.

The directional derivative for the function $F$ at $x_0=0$ is equal  
$$
\lim\limits_{k\rightarrow \infty }\frac{F(t_kh)}{t_k}= \lim\limits_{k\rightarrow \infty } \frac{i(\bar{F}(t_kh))}{t_k} =\lim\limits_{k\rightarrow \infty } i\left( \frac{\bar{F}(t_kh)}{t_k} \right) =\lim_{k\rightarrow \infty} i(y_k).
$$
By the fact that $i$ is an isomorphism and by Proposition \ref{izomnorm}, cone $i(K)$ is closed normal and generating in $Z$.
Since $\{y_k\}$ is not weakly convergent, the function $F$ is $i(K)$-convex  and is not directionally differentiable at $x_0=0$ in the direction $0\neq h \in X.$
\end{proof}
Since $\{b_i\}$ is an unconditional basis for $c_0$, from Theorem \ref{mc} we get the following corollary.
\begin{corollary}
\label{cor1}
Let $X$ be a linear space and $Y$ be a Banach space.  If for every  closed convex  normal and generating cone $K\subset Y$ and 
 for  every $K$-convex function $F:X\rightarrow Y$ there exist directional derivative for every $h\in X$ at $x_0=0,$ then $Y$ is not isomorphic to $c_0.$
\end{corollary}
If $Y$ is weakly sequentially complete Banach space, then $K$-convex function, where $K\subset Y$ is closed convex and normal has directional derivative for every $x_0 \in X,h \in X\setminus\{0\}$ (see \cite{valadier}).
From Theorem \ref{dlakrat} and Corollary \ref{cor1} we get the characterization of weakly sequentially complete Banach spaces in terms of existence of directional derivative for $K$-convex functions.
\begin{theorem}
Let $Y$ be a Banach lattice. Space $Y$ is weakly sequentially complete if and only if, for every closed convex normal cone $K\subset Y$ and every $K$-convex function $f:X\rightarrow Y$ the directional derivative exist for all $x_0\in X, h\in X\setminus \{0\}.$ 
\end{theorem}

\end{document}